\documentclass[reqno]{amsart}
\usepackage{latexsym,amsmath,amssymb
}
\usepackage{amsthm}
\usepackage{enumerate}
\usepackage{hyperref} 
\makeatletter
\@addtoreset{figure}{section}
\def\thefigure{\thesection.\@arabic\c@figure}
\def\fps@figure{h,t}
\@addtoreset{table}{bsection}

\def\thetable{\thesection.\@arabic\c@table}
\def\fps@table{h, t}
\@addtoreset{equation}{section}

\makeatother

\DeclareMathAlphabet\mathbfcal{OMS}{cmsy}{b}{n}

\newtheorem{theorem}{Theorem}
\newtheorem{corollary}[theorem]{Corollary}
\newtheorem{definition}[theorem]{Definition}

\newtheorem{lemma}[theorem]{Lemma}

\newtheorem{proposition}[theorem]{Proposition}
\newtheorem{remark}[theorem]{Remark}
\newtheorem{hypothesis}[theorem]{Hypothesis}

\numberwithin{theorem}{section}
\numberwithin{equation}{section}

\newcommand{\1}{{\bf 1}}

\newcommand{\alg}{{\rm alg}}
\newcommand{\bt}{{\bf t}}
\newcommand{\card}{{\rm card}\,}
\newcommand{\de}{{\rm d}}
\newcommand{\ee}{{\rm e}}
\newcommand{\Gr}{{\rm Gr}}
\newcommand{\Hom}{{\rm Hom}}
\newcommand{\id}{{\rm id}}
\newcommand{\ie}{{\rm i}}
\newcommand{\Ind}{{\rm Ind}}
\newcommand{\jump}{{\rm jump}\,}
\newcommand{\Ker}{{\rm Ker}\,}
\newcommand{\LC}{{\rm LC}}
\newcommand{\Subquot}{{\rm SQ}}
\newcommand{\spa}{{\rm span}\,}
\newcommand{\supp}{{\rm supp}\,}
\newcommand{\Tr}{{\rm Tr}\,}

\newcommand{\dual}[2]{\langle #1, #2\rangle}
\newcommand{\norm}[1]{\Vert #1 \Vert }
\newcommand{\scalar}[2]{(#1 \mid#2) }

\newcommand{\CC}{{\mathbb C}}
\newcommand{\KK}{{\mathbb K}}
\newcommand{\NN}{{\mathbb N}}
\newcommand{\RR}{{\mathbb R}}
\newcommand{\TT}{{\mathbb T}}

\newcommand{\Ac}{{\mathcal A}}
\newcommand{\Bc}{{\mathcal B}}
\newcommand{\Cc}{{\mathcal C}}
\newcommand{\Dc}{{\mathcal D}}
\newcommand{\Hb}{{\mathbfcal H}}
\newcommand{\Hc}{{\mathcal H}}
\newcommand{\Ic}{{\mathcal I}}
\newcommand{\Jc}{{\mathcal J}}
\newcommand{\Kc}{{\mathcal K}}
\newcommand{\Sc}{{\mathcal S}}
\newcommand{\Tc}{{\mathcal T}}
\newcommand{\Uc}{{\mathcal U}}
\newcommand{\Vc}{{\mathcal V}}
\newcommand{\Wc}{{\mathcal W}}
\newcommand{\Xc}{{\mathcal X}}
\newcommand{\Yc}{{\mathcal Y}}

\newcommand{\Ig}{{\mathfrak I}}
\newcommand{\Pg}{{\mathfrak P}}
\newcommand{\Sg}{{\mathfrak S}}

\renewcommand{\gg}{{\mathfrak g}}

\newcommand{\kg}{{\mathfrak k}}
\newcommand{\mg}{{\mathfrak m}}
\newcommand{\pg}{{\mathfrak p}}

\title[$C^*$-algebras of completely solvable Lie groups]{The $C^*$-algebras of completely solvable Lie groups are solvable}

\author{Ingrid Belti\c t\u a}
\author{Daniel Belti\c t\u a}
\address{Institute of Mathematics ``Simion Stoilow'' of the Romanian Academy,
	P.O. Box 1-764, Bucharest, Romania}

\email{Ingrid.Beltita@imar.ro, ingrid.beltita@gmail.com}
\email{Daniel.Beltita@imar.ro, beltita@gmail.com}
\keywords{completely solvable Lie group; solvable $C^*$-algebra}
\subjclass[2020]{Primary 22E27; Secondary 17B30, 46L05, 46L55}
\thanks{We acknowledge financial
	support from the Research Grant GAR 2023 (code 114), supported from the Donors’
	Recurrent Fund of the Romanian Academy, managed by the ”PATRIMONIU”
	Foundation.}

\dedicatory{To Karl-Hermann Neeb on the occasion of his 60th birthday}

\begin{document}
	
	\parskip=5pt

\begin{abstract}
	We prove that if  a connected and simply connected  Lie group $G$ admits connected closed normal subgroups $G_1\subseteq G_2\subseteq \cdots \subseteq G_m=G$ with $\dim G_j=j$ for $j=1,\dots,m$, then its group $C^*$-algebra has  closed two-sided ideals $\{0\}=\Jc_0\subseteq \Jc_1\subseteq\cdots\subseteq\Jc_n=C^*(G)$ 
	with $\Jc_j/\Jc_{j-1}\simeq \Cc_0(\Gamma_j,\Kc(\Hc_j))$ 
	for a suitable locally compact Hausdorff space $\Gamma_j$ and a separable complex Hilbert space $\Hc_j$, where $\Cc_0(\Gamma_j,\cdot)$ denotes the continuous mappings on $\Gamma_j$ that vanish at infinity, and $\Kc(\Hc_j)$ is the $C^*$-algebra of compact operators on~$\Hc_j$ for $j=1,\dots,n$. 
	\end{abstract}

	\maketitle

	\section{Introduction}

	In this paper we establish a direct relation between solvability properties of two very different mathematical objects: 
	\begin{enumerate}[1.]
	\item 
	A  \emph{completely solvable}
 Lie group $G$, i.e., a connected and simply connected (real) Lie group
	whose Lie algebra $\gg$ has a chain of ideals $\{0\}=\gg_0\subseteq\gg_1\subseteq\cdots\subseteq\gg_m=\gg$ with $\dim(\gg_j/\gg_{j-1})=1$, $j=1,\dots,m$. 
	Clearly $m=\dim\gg$. 
	\item
	A \emph{solvable} $C^*$-algebra $\Ac$, i.e., 
	a $C^*$-algebra that has a finite chain of closed two-sided ideals $\{0\}=\Jc_0\subseteq\Jc_1\subseteq\cdots\subseteq\Jc_n=\Ac$ 
	with $\Jc_j/\Jc_{j-1}\simeq  \Cc_0(\Gamma_j, \Kc(\Hc_j))$ 
	for suitable locally compact Hausdorff spaces $\Gamma_j$ and separable complex Hilbert space $\Hc_j$, $j=1,\dots,n$.
	The least number $n$ for which there exists a chain of ideals as above is called the \emph{length} of~$\Ac$. 
	\end{enumerate}
	Here $\Cc_0(\Gamma_j,\cdot)$ denotes the continuous mappings on $\Gamma_j$ that vanish at infinity, and $\Kc(\Hc_j)$ is the $C^*$-algebra of compact operators on~$\Hc_j$. 
	With the above terminology, one of the main results of this paper (Corollary~\ref{solv}) is that if a Lie group $G$ is completely solvable, 
	then its group $C^*$-algebra~$C^*(G)$ is solvable. 
	In the special case of nilpotent Lie groups, that result was obtained in \cite[Th. 4.11]{BBL16} by a  method that required global canonical coordinates on coadjoint orbits. (See also \cite{GK22}.)
	In contrast, the present approach relies on the continuous selections of polarizations constructed in \cite{BB24} along with the technique of ultrafine layerings
	initiated in \cite{CuPn89} and \cite{Cu92b}. 
	(See also \cite{ArCu20} and \cite{ArCu24}.)
	
	This result is remarkable since,  
 	as already noted in \cite{Dy78}, the $C^*$-algebra of a solvable Lie
 	 group may not be solvable, a fact
	 further discussed in \cite{Pe85}. 
	A weaker condition, in which the successive quotients $\Jc_j/\Jc_{j-1}$ should be continuous-trace $C^*$-algebras, was verified in  \cite{Cu92b} for the wider class of exponential Lie groups. 
	
	Even in the special case when $G$ is a nilpotent Lie group,  the relation between the length of $C^*(G)$ 
	and various algebraic data of $G$, such as the nilpotency step, remains to be clarified.
For example,
the length of~$C^*(G)$ is equal to~1 if and only if $G$ is commutative, irrespective of the dimension of~$G$. 
Nontrivial examples already appear for 2-step nilpotent Lie groups: 
for every positive integer $n\ge 1$, if $H_{2n+1}$ denotes the $(2n+1)$-dimensional Heisenberg group, then the length of $C^*(H_{2n+1})$ is equal to~2. 
(See for instance \cite{BB21}, \cite{BB24_RIMS}, and the references therein.)
Moreover, if $G$ is the free 2-step nilpotent Lie group with 3 generators, then $\dim G=6$ while the length of $C^*(G)$ is again equal to~2, cf. \cite[Ex. 6.3.5]{Ech96} and also \cite[Ex. 7.2]{BB24}. 
Other examples can be found in 	\cite{BB21}.

	The structure of the present paper is as follows: 
	Section~\ref{Sect_prel} contains preliminaries on the Pedersen ideal of a $C^*$-algebra and continuous selections of Haar measures on the closed subgroups of a locally compact group. 
	Section~\ref{Sect_fields} contains a construction of continuous fields of Hilbert spaces and our main technical result (Theorem~\ref{e-cont-field}) gives sufficient conditions 
	for a con\-tin\-u\-ous-trace subquotient of $C^*(G)$ to be Morita-equivalent to a commutative $C^*$-algebra if $G$ is an exponential Lie group. 
	This result is motivated by the continuous-trace subquotient conjecture \cite[Conj. 4.18]{RaRo98} to the effect that every continuous-trace subquotient of $C^*(G)$ is Morita-equivalent to a commutative $C^*$-algebra, which is still an open problem. 
    (See also \cite[Rem. 3.3]{Cu92b}.) 
    Finally, Section~\ref{section4} includes our main results (Theorem~\ref{finetriv} and Corollary~\ref{solv}).

	\subsection*{General notation}
	We denote the connected and simply connected  Lie groups by upper case Roman letters and their Lie algebras by the corresponding
	lower case Gothic letters. 
	By a solvable/nilpotent Lie group we always mean
	a \textit{connected and simply connected} solvable/nilpotent Lie group, unless otherwise specified.
	
	An exponential Lie group is a Lie group $G$ whose exponential map
	$\exp_G \colon\gg \to G$ is bijective. 
	All exponential Lie groups are solvable, and completely solvable Lie groups are exponential. 
	(See for instance \cite[\S 14.4]{HN12}.)
	We emphasize that in this paper, all completely solvable Lie groups are connected and simply connected.

	For any Lie algebra $\gg$ with its linear dual space $\gg^*$ we denote by $\langle\cdot,\cdot\rangle\colon\gg^*\times\gg\to\RR$ the corresponding duality pairing.  
	We often denote the group actions simply by juxtaposition, and in particular this is the case for the coadjoint action $G\times\gg^*\to\gg^*$, $(g,\xi)\mapsto g\xi$. 
	
	For every $C^*$-algebra $\Ac$ we denote by $\Subquot^\Tr(\Ac)$ the set of its $*$-isomorphism classes of continuous-trace subquotients, that is, $*$-isomorphism classes of the $C^*$-algebras $\Jc_2/\Jc_1$ which have continuous trace, where $\Jc_1\subseteq\Jc_2$ are closed two-sided ideals of $\Ac$. 
	By abuse of notation, we write $\Sc\in \Subquot^\Tr(\Ac)$ 
	for $\Sc:=\Jc_2/\Jc_1$ as above. 
	We refer to \cite[Rem. 2.4]{BBL16} for the bijection between subquotients and locally closed subsets of the dual space $\widehat{\Ac}$ consisting of equivalence classes of irreducible $*$-representations of~$\Ac$. 
	We denote by $\LC^{\Tr}(\widehat{\Ac})$ the set of all locally closed subsets of $\widehat{\Ac}$ which are dual spaces of continuous-trace subquotients of $\Ac$. 
	If $\Ac=C^*(G)$ for a solvable Lie group $G$, then we write $\LC^{\Tr}(\widehat{G})$ instead of $\LC^{\Tr}(\widehat{\Ac})$, because of the canonical identification of $\widehat{\Ac}$ with the unitary dual $\widehat{G}$ consisting of equivalence classes of irreducible unitary representations of~$G$.

	For every unitary representation $\pi$ we denote by $[\pi]$ its unitary equivalence class. 
	We use similar notation for $*$-representations of associative Banach $*$-algebras, in particular for $C^*$-algebras. 
	Our general references for $C^*$-algebras and Morita equivalence are \cite{Di64} and \cite{RW98}, 
	and we refer to \cite{Pu68}, \cite{BCDLRRV72}, and \cite{ArCu20} for representation theory of exponential Lie groups.

	\section{Preliminaries}
	\label{Sect_prel}
	
	\subsection{On the smallest dense ideal of a continuous-trace $C^*$-algebra}
	
	\begin{definition}
		\normalfont 
		Let $\Ac$ be any $C^*$-algebra with $\Ac^{+}:=\{a\in \Ac\mid 0\le a\}$. 
		We define $K_0^{+}$ as the set of all $a\in \Ac^{+}$ for which there exist
		$b\in \Ac^{+}$ and a 
		continuous function $\varphi\colon (0, \infty) \to [0, \infty)$ with compact support
		such that $a=\varphi(b)$.
		We also define 
		$$
		K_\Ac^{+} :=\{a\in \Ac^{+}\mid (\exists n\in\NN)(\exists a_1,\dots,a_n\in K_0^{+})\quad a\le a_1+\cdots+a_n\}.$$
		Then $K_\Ac :=\spa K_\Ac^{+}$ is called the \emph{Pedersen ideal} of the $C^*$-algebra $\Ac$. 
	\end{definition}
	
	\begin{definition}
		\normalfont 
		For any $C^*$-algebra $\Ac$ we denote by $J_\Ac$ 
		the set of all $a\in \Ac$ 
		for which $\sup\{\dim\pi(a)<\infty\mid [\pi]\in\widehat{\Ac}\}$ 
		and the function $[\pi]\mapsto \Tr\pi(a)$ 
		vanishes outside some quasi-compact set. 
	\end{definition}
	
	\begin{definition}
		\normalfont 
		For any $C^*$-algebra $\Ac$ we denote by $\mg_\Ac^{+}$ 
		the set of all $a\in \Ac^{+}$ 
		for which the function $[\pi]\mapsto \Tr\pi(a)$ 
		is continuous and finite on $\widehat{\Ac}$. 
		We also denote by $\mg_\Ac$ the linear span of $\mg_\Ac^{+}$, 
		which is a self-adjoint two-sided ideal of~$A$ with $\mg_A\cap \Ac^{+}=\mg_\Ac^{+}$. 
	\end{definition}
	
	\begin{proposition}\label{PI}
		The following assertions hold for any $C^*$-algebra $\Ac$: 
		\begin{enumerate}[{\rm(i)}]
			\item\label{PI_item0} 
			The set $J_\Ac$ is a self-adjoint two-sided ideal of~$\Ac$. 
			\item\label{PI_item1} 
			The set $K_\Ac$ is the smallest element of the set of all dense self-adjoint two-sided ideals of $\Ac$, 
			 one has $K_\Ac\cap \Ac^{+}=K_\Ac^{+}$, and moreover $K_\Ac^{+}$ is dense in $\Ac^+$. 
			\item\label{PI_item2} 
			If $\Phi\colon \Ac\to \Bc$ is any surjective $*$-morphism of $C^*$-algebras, then $\Phi(K_\Ac)=K_\Bc$ and $\Phi(K_\Ac^+)=K_\Bc^+$. 
			\item\label{PI_item3} 
			If $\Ac$ has continuous trace and every compact subset of $\widehat{\Ac}$ has finite covering dimension, 
			then $K_\Ac=J_\Ac$. 
			\item\label{PI_item4} 
			The $C^*$-algebra $\Ac$ has continuous trace if and only if $K_\Ac\cap \Ac^{+}\subseteq \mg_\Ac^{+}$.  
			\item\label{PI_item5} 
			If $\Ac$ has continuous trace and every compact subset of $\widehat{\Ac}$ has finite covering dimension, 
			then $J_\Ac\cap \Ac^{+}\subseteq \mg_\Ac^{+}$. 
			\end{enumerate}
	\end{proposition}
	
	\begin{proof}
		Assertion~\eqref{PI_item0} is a direct consequence of the definition of $J_\Ac$.
		
		Assertion~\eqref{PI_item1} follows by \cite[Th.~5.6.1 and its proof]{P79}.

		Assertion~\eqref{PI_item2} follows by  the previous assertion, since $\Phi(K_\Ac^+)= K_\Bc^+$.

		Assertion~\eqref{PI_item3} follows by \cite[Cor.2.8]{GT79} along with the alternative description 
		of the Pedersen ideal given in Assertion~\eqref{PI_item1} above. 
		
		Assertion~\eqref{PI_item4} is a by-product of the proof of \cite[Th. 17]{Gr77}. 
		Specifically, if $\Ac$ has continuous trace, then $\mg_\Ac$ is dense in $\Ac$, hence $K_\Ac\subseteq \mg_\Ac$ by Assertion~\eqref{PI_item1}, 
		and then $K_\Ac\cap \Ac^{+} \subseteq \mg_\Ac \cap \Ac^{+}=\mg_\Ac^{+}$. 
		Conversely, assume that $K_\Ac\cap \Ac^{+}\subseteq \mg_\Ac^{+}$. 
		Since $K_\Ac$ is spanned by its positive part $K_\Ac\cap \Ac^{+}$, 
		we obtain $K_\Ac\subseteq \mg_\Ac$. 
		We have that $K_\Ac$ is dense in $\Ac$ by Assertion~\eqref{PI_item1} above, 
		hence also $\mg_\Ac$ is dense in $\Ac$, that is, $\Ac$ has continuous trace. 
		
		Assertion~\eqref{PI_item5} follows by Assertions \eqref{PI_item3}--\eqref{PI_item4}. 
	\end{proof}

\begin{corollary}
\label{isom}
Let $\Sc$ be a continuous-trace separable  $C^*$-algebra, denote $\Gamma:=\widehat{\Sc}$,
and assume the following: 
\begin{itemize}
	\item ${\Hb}_\Gamma = (\Hc_\gamma)_{\gamma \in \Gamma}$ is a continuous field of Hilbert spaces  with its corresponding continuous field of elementary $C^*$-algebras  $\Ac(\Hb_\Gamma)= (\Kc(\Hc_\gamma))_{\gamma \in \Gamma}$ and continuous-trace $C^*$-algebra of global sections~$\Ac_0(\Hb_\Gamma)$.
	\item $\pi_\gamma\colon\Sc\to\Bc(\Hc_\gamma)$ is an irreducible $*$-representation with $[\pi_\gamma]=\gamma$ for every $\gamma\in\Gamma$. 
	\item For every $a\in K_\Sc^+$ we have $(\pi_\gamma(a))_{\gamma\in\Gamma}\in \Ac_0(\Hb_\Gamma)$. 
	\end{itemize} 
Then the mapping $\Sc\to \Ac_0(\Hb_\Gamma)$, $a\mapsto(\pi_\gamma(a))_{\gamma\in\Gamma}$, 
is a well-defined $*$-isomorphism.
\end{corollary}

\begin{proof}
We consider the $*$-morphism
$$\Phi\colon \Sc\to \prod\limits_{\gamma}\Kc(\Bc_\gamma),\quad \Phi(a):=(\pi_\gamma(a))_{\gamma\in\Gamma}$$
which takes values as indicated, since $\Sc$ is a liminal $C^*$-algebra by \cite[Prop. 4.5.3(i)]{Di64}. 
Moreover, we have $\Ker\Phi=\{0\}$ since $\widehat{\Sc}=\{[\pi_\gamma]\mid\gamma\in\Gamma\}$. 

It remains to show that $\Phi(\Ac)=\Ac_0(\Hb_\Gamma)$. 
To this end we recall from Proposition~\ref{PI}\eqref{PI_item1} that $K_\Sc^+$ is dense in $\Sc^+$, while $\Phi(K_\Sc^+)\subseteq \Ac_0(\Hb_\Gamma)$ by hypothesis, 
hence $\Phi(\Sc^+)\subseteq \Ac_0(\Hb_\Gamma)$. 
Since every element of the $C^*$-algebra $\Sc$ is a linear combination of positive elements, 
we then obtain $\Phi(\Sc)\subseteq \Ac_0(\Hb_\Gamma)$. 

We now recall that $\Ac_0(\Hb_\Gamma)$ is a liminal $C^*$-algebra whose spectrum is equal to~$\Gamma$, cf. \cite[10.9.2]{Di64}, and for every $\gamma\in \Gamma$, $\Ac_0(\Hb_\Gamma)\ni a\mapsto a(\gamma) \in \Kc(\Hc_\gamma)$ is the unique irreducible representation, modulo unitary equivalence, corresponding to $\gamma$. 
It follows that the $C^*$-subalgebra $\Phi(\Sc)\subseteq \Ac_0(\Hb_\Gamma)$ is rich, hence $\Phi(S)= \Ac_0(\Hb_\Gamma)$ by the noncommutative Stone-Weierstrass theorem \cite[Lemma 11.1.4]{Di64}. 
\end{proof}

\subsection{Chabauty-Fell topology on closed subgroups, Haar measures}
	
	Let $G$ be a locally compact group with its unit element denoted by $\1\in G$ and its set of closed subgroups denoted by~$\Sigma(G)$. 
	The set $\Sigma(G)$ is regarded as a compact Hausdorff topological space, 
	endowed with its Chabauty-Fell topology (cf. \cite[Prop. 2(vi)]{dlH08} or \cite[Cor. H.4]{Wi07}). 
	We recall that a basis of open sets of the Chabauty-Fell topology consists of the sets 
	$$\Uc(C, \Tc)= \{ K\in \Sigma(G)\mid K\cap C=\emptyset, \; K \cap A\ne \emptyset, \forall
	\, A\in \Tc\},$$
	where $C$ runs over the compact subspaces of $G$ and $\Tc$ over finite families of nonempty open 
	subsets of $G$. 
We also define 
$$\Sigma_0(G):=\{K\in\Sigma(G)\mid K\text{ is connected}\}$$ 
and
$$\Tc(G):=
\{(K,g)\in\Sigma(G)\times G\mid g\in K\}.$$
One can prove that $\Tc(G)$ is a closed subset of $\Sigma(G)\times G$. 

	\begin{remark}\label{eqtop}
		\normalfont
	Assume that $G$ is an exponential Lie group with its Lie algebra $\gg$. 
	Let $\Gr(\gg)$ be the Grassmann manifold of $\gg$, whose points are the linear subspaces of $\gg$. 
	We also consider the subset $\Gr_\alg(\gg)$ whose points are the subalgebras of the Lie algebra~$\gg$. 
	The smooth manifold $\Gr(\gg)$ is compact and 
	$\Gr_\alg(\gg)$ is a closed subset, hence $\Gr_\alg(\gg)$  is in turn a compact space. 
	(See \cite[Sect. 6]{BB24}.) 
Then the mapping $\Gr_\alg(\gg)\to\Sigma(G)$, $\kg\mapsto\exp_G(\kg)$, 
is a homeomorphism from $\Gr_\alg(\gg)$ onto $\Sigma_0(G)$. 
		(See e.g., \cite[Lemma~2.4]{GS99}.)
	\end{remark}

	We denote by $\Cc_c(G)$ the set of $\CC$-valued continuous functions on $G$ having compact support and we fix a function $0\le\varphi_0\in\Cc_c(G)$ with $\varphi_0(\1)=1$. 
	
	For every $K\in\Sigma(G)$ we introduce the following notation, cf. \cite[Sect. C.1]{RW98}: 
	\begin{itemize}
		\item  $\mu_K$ is the left-invariant Haar measure of $K$ satisfying $\int\limits_K\varphi_0(k)\de\mu_K(k)=1$. 
		\item $\Delta_K\colon K\to\RR_+^\times$ is the modular function of $K$, that is, 
		the continuous group homomorphism satisfying 
		\begin{align}
		\label{right}
		\int\limits_K\psi(k)\de\mu_K(k)& =\Delta_K(k_0)\int\limits_K\psi(kk_0)\de\mu_K(k), \\
		\label{inv}
		\int\limits_K\psi(k)\de\mu_K(k)& =\int\limits_K\psi(k^{-1})\Delta_K(k^{-1})\de\mu_K(k)
	\end{align}
for every $k_0\in K$ and $\psi\in\Cc_c(K)$. 
		\item $\omega_K\colon G\to \RR_+^\times$, $\omega_K(g):=\int\limits_K\varphi_0(gkg^{-1})\de\mu_K(k)$. 
		Note that $\omega_K\vert_{K}= \Delta_K$, and $\omega_K(gh) =\omega_K(g) \omega_K(h)$ for every $g\in G$ and $h \in K$. 
		\item $\rho_K\colon G\to \RR_+^\times$, $\rho_K(g):=\Delta_G(g^{-1})\omega_K(g)$. 
		Note that $\rho_K\vert_{K}= \Delta_K (\Delta_G\vert_{K})^{-1}$, and 
		$\rho_K(gh) =\rho_K(g) \rho_K(h)$ for every $g\in G$ and $h \in K$. 
		\item $\nu_K$ is the positive Radon measure on $G/K$ satisfying 
		\begin{equation}
		\label{rho_eq0}
		\int\limits_G\varphi(g)\rho_K(g)\de\mu_G(g)
		=\int\limits_{G/K}\Bigl(\int\limits_K \varphi(gk)\de\mu_K(k)\Bigr)\de\nu_K(gK)
		\end{equation}
	for every $\varphi\in\Cc_c(G)$. 
	\end{itemize}
We call the above family of measures $(\mu_K)_{K\in\Sigma(G)}$  a \emph{Haar system} on the closed subgroups of~$G$ and for every $K\in \Sigma(G)$ we regard $\mu_K$ as a Radon measure on~$G$ supported by~$K$.
 
We now give a version of \cite[Lemma H.9]{Wi07} adapted to our purposes in the present paper. 
For any locally compact Hausdorff space $X$ we endow $\Cc_c(X)$ 
with the inductive limit topology as inductive limit of the Banach spaces $\Cc_C(X):=\{f\in\Cc_c(X)\mid \supp f\subseteq C\}$ for arbitrary compact $C\subseteq X$. 
(See e.g., \cite[Ch. 4, Part 1, \S 2, Ex. (d)]{Gro73}.)

\begin{lemma}
	\label{Wi07_LH.9}
	If $G$ is a locally compact group and $(\mu_K)_{K\in\Sigma(G)}$ is a Haar system on the closed subgroups of $G$, then 
	the following assertions hold:  
	\begin{enumerate}[{\rm(i)}]
		\item\label{Wi07_LH.9_item1}
		The mapping 
		$$\Ic\colon \Cc_c(G)\to\Cc(\Sigma(G)),\quad (\Ic(f))(K):=\int f\de\mu_K$$
		is well defined, linear, and continuous. 
		\item\label{Wi07_LH.9_item2}
		For every compact subset $C\subseteq G$ 
		the mapping 
		$$\widehat{\Ic}_C\colon \Sigma(G)\times\Cc_C(G)\to\CC,\quad 
		\widehat{\Ic}_C(K,f):=\int f\de\mu_K$$
		is continuous. 
		\item\label{Wi07_LH.9_item3} 
		The subset $Y:=\{(K,k)\in\Sigma(G)\times G\mid k\in K\}\subseteq \Sigma(G)\times G$ is closed and 
		the function 
		$$\widetilde{\Delta}\colon Y\to\CC,\quad \widetilde{\Delta}(K,k):=\Delta_K(k)$$
		is continuous. 
		\item\label{Wi07_LH.9_item4} 
		If $X$ is a locally compact space, then for every $F\in\Cc_c(X\times\Sigma(G)\times G)$ its corresponding function 
		\begin{equation}
			\label{Wi07_LH.9_item3_eq1} 
			\widetilde{F}\colon X\times \Sigma(G)\to\CC,\quad \widetilde{F}(x,K):=\int F(x,K,k)\de\mu_K(k)
		\end{equation}
		is continuous. 
	\end{enumerate}
\end{lemma}

\begin{proof}
	\eqref{Wi07_LH.9_item1} 
	For every $f\in\Cc_c(G)$ we have $\Ic(f)\in\Cc(\Sigma(G))$  by 
	 \cite[Appendix]{Gl62} or \cite[Lemma H.8]{Wi07},  
	hence the mapping $\Ic$ in the statement is well defined. 
	Moreover, the mapping $\Ic$ is clearly linear. 
	Therefore, by \cite[Ch. 4, Part 1, \S 1, Prop. 2]{Gro73}, it suffices to prove that for an arbitrary compact subset $C\subseteq X$ the linear map 
	$\Ic\vert_{\Cc_C(G)}\colon\Cc_C(G)\to\Cc(\Sigma(G))$ is continuous. 
	
	To this end we first select $0\le \varphi\in\Cc_c(G)$ with 
	$\varphi\vert_C=1$. 
	Then for every $f\in\Cc_C(G)$ we have 
	\begin{equation*}
		\vert f\vert \le \Vert f\Vert_\infty\varphi.
	\end{equation*}
	This further implies 
	\begin{equation*}
		\vert (\Ic(f))(K)\vert=\Bigl\vert\int f\de\mu_K\Bigr\vert
		\le \int\vert f\vert \de\mu_K
		\le \Vert f\Vert_\infty \int\varphi\de\mu_K
		=\Vert f\Vert_\infty  (\Ic(\varphi))(K). 
	\end{equation*}
	Now, since $0\le \Ic(\varphi)\in\Cc(\Sigma(G))$ and  the space $\Sigma(G)$ is compact,  
	it follows that 
	$$\Vert \Ic(\varphi)\Vert_\infty=\sup\limits_{K\in\Sigma(G)}(\Ic(\varphi))(K)<\infty$$
	hence the above inequalities imply 
	\begin{equation}
		\label{Wi07_LH.9_proof_eq1}
		(\forall f\in\Cc_C(G))\quad \Vert \Ic(f)\Vert_\infty\le 
		\Vert \Ic(\varphi)\Vert_\infty\Vert f\Vert_\infty. 
	\end{equation}
	This shows that the linear map 
	$\Ic\vert_{\Cc_C(G)}\colon\Cc_C(G)\to\Cc(\Sigma(G))$ is continuous.

	\eqref{Wi07_LH.9_item2} 
	It is enough to prove that if $\lim\limits_{j\in J}K_j=K$ in $\Sigma(G)$ and $\lim\limits_{j\in J}f_j=f$ in $\Cc_C(G)$ then $\lim\limits_{j\in J}\widehat{\Ic}_C(K_j,f_j)=\widehat{\Ic}_C(K,f)$. 
	For arbitrary $j\in J$ we have 
	\begin{align*}
		\vert\widehat{\Ic}_C(K_j,f_j)-\widehat{\Ic}_C(K,f)\vert
		& \le\vert\widehat{\Ic}_C(K_j,f_j-f)\vert
		+\vert \widehat{\Ic}_C(K_j,f)-\widehat{\Ic}_C(K,f)\vert \\
		&=\vert (\Ic(f_j-f))(K_j)\vert 
		+\vert (\Ic(f))(K_j)-(\Ic(f))(K)\vert \\
		&\le 	\Vert \Ic(\varphi)\Vert_\infty\Vert f_j-f\Vert_\infty
		+\vert (\Ic(f))(K_j)-(\Ic(f))(K)\vert, 
	\end{align*}
	where we used \eqref{Wi07_LH.9_proof_eq1}. 
	Now, using the fact that $\lim\limits_{j\in J}f_j=f$ in $\Cc_C(G)$ and $\Ic(f)\in\Cc(\Sigma(G))$, we obtain  $\lim\limits_{j\in J}\widehat{\Ic}_C(K_j,f_j)=\widehat{\Ic}_C(K,f)$.

	\eqref{Wi07_LH.9_item3}
	See \cite[Lemma H.9(c)]{Wi07}. 
	
	\eqref{Wi07_LH.9_item4} 
	Since $F\in\Cc_c(X\times\Sigma(G)\times G)$, there exists compact subsets $X_0\subseteq X$ and $C\subseteq G$ with $\supp F\subseteq X_0\times\Sigma(G)\times C$. 
	This implies that the function 
	$$\widehat{F}\colon X\times\Sigma(G)\to\Cc_c(G),\quad \widehat{F}(x,K):=F(x,K,\cdot)$$
	is well defined and continuous. 
	Then, by \eqref{Wi07_LH.9_item2} above, the function 
	\begin{align*}\widehat{\Ic}_C\circ(\id_{\Sigma(G)}\times \widehat{F}) 
		& \colon X\times\Sigma(G)\to\CC,\\
		(\widehat{\Ic}_C\circ(\id_{\Sigma(G)}\times \widehat{F})) 
		& (x,K)=\widehat{\Ic}_C(K,\widehat{F}(x,K))
	\end{align*}
	is continuous as a composition of two continuous maps. 
	Since 
	$$\widehat{\Ic}_C(K,\widehat{F}(x,K))=\int F(x,K,k)\de\mu_K(k)
	=\widetilde{F}(x,K),$$
	and this completes the proof. 
\end{proof}

\section{Continuous fields of Hilbert spaces defined by continuous selections of polarizations}
\label{Sect_fields}

In this section we obtain our main technical results that give sufficient conditions for triviality of certain continuous-trace subquotients of group $C^*$-algebras (Theorem~\ref{e-cont-field} and Corollary~\ref{cont-field}). 

\subsection{Continuous fields of Hilbert spaces defined by continuous selections of subgroups}

Let $G$ be a locally compact group.

\begin{hypothesis}\label{e-cont-hyp}
\normalfont 
Throughout this subsection we assume the following:
	\begin{enumerate}[{\rm(i)}]
		\item\label{cont-hyp_item1}   
		$\Gamma\in \LC^{\Tr}(\widehat{G})$, $\Gamma$
has finite covering dimension, and every compact subset of $\Gamma$ has finite covering dimension. 
		\item\label{cont-hyp_item3}   
		We have a continuous  mapping  $\Gamma\to \Sigma(G)$, $\gamma\mapsto P(\gamma)$. 
		 \item\label{cont-hyp_item2}
		For every $\gamma\in \Gamma$, $\chi_\gamma\colon G\to \TT$ is such that $\chi_\gamma\vert_{P(\gamma)}$ is a character and 
		the map $\Gamma\times G\ni (\gamma, g) \mapsto \chi_\gamma(g)\in \TT $ is continuous. 
		\item\label{cont-hyp_item4}   
		For every $\gamma\in \Gamma$, $\pi_\gamma=\Ind_{P(\gamma)}^G \chi_\gamma$ is an infinite-dimensional  irreducible representation and 
		$\gamma=[\pi_\gamma]$.
	\end{enumerate}  
\end{hypothesis}
\begin{remark}\label{paracomp}
\normalfont
In Hypothesis~\ref{e-cont-hyp}, the set
 $\Gamma$ is the spectrum of a continuous-trace separable $C^*$-algebra, 
therefore it is a second countable, locally compact Hausdorff space. 
(See \cite[3.3.4, 3.3.8, and 4.5.3]{Di64}.)
By \cite[Lemma 6.5]{Wi07},
it follows that $\Gamma$ is a (completely) metrizable space, therefore
 paracompact. 
\end{remark}

For every $\gamma\in\Gamma$ 
we set 
$$\mu_\gamma:=\mu_{P(\gamma)},\  \nu_\gamma:= \nu_{G/P(\gamma)},\ 
\rho_\gamma:= \rho_{P(\gamma)},$$ 
with 
$\mu_{P(\gamma)}$, $\nu_{G/P(\gamma)}$, and $ \rho_{P(\gamma)}$ as in Section~\ref{Sect_prel}.

Define
\begin{align*}
	\Dc(G,P(\gamma), \chi_\gamma) = \{ \varphi\in \Cc(G)\mid 
	&
	(\forall g\in G, h\in P(\gamma))\quad  
	\varphi(gh) =\chi_{\gamma} (h)^{-1}\varphi(g),    \\
	&q_\gamma(\supp\,f) \text{ compact}\}, 
\end{align*}
where $q_\gamma\colon G\to G/P(\gamma)$, $g\mapsto \dot g$, is the quotient map, 
and consider the norm 
$$ \norm{f}_\gamma :=\Big(\int\limits_{G/P(\gamma)}\vert\varphi(g)\vert^2\de \nu_\gamma(\dot g)\Big)^{1/2}.
$$
The Hilbert space $\Hc_\gamma$, defined
as the completion of $\Dc(G,P(\gamma), \chi_\gamma)$ with respect to the norm 
$\norm{\cdot}_\gamma$, 
is the representation space of the induced representation 
$\pi_\gamma$.

We now consider the linear mapping $R_\gamma\colon \Cc_c(G) \to \Dc(G,P(\gamma), \chi_\gamma)$ 
given by 
\begin{equation}\label{R1}
 (R_\gamma \varphi)(g):= \int\limits_{P(\gamma)} \rho_\gamma(gh)^{-1/2}\varphi(gh) \chi_\gamma(h) \de \mu_\gamma(h).
\end{equation}
We note that, if $k\in P(\gamma)$ and $g\in G$, then 
\begin{align}
(R_\gamma \varphi)(gk)
&= \int\limits_{P(\gamma)} \rho_\gamma(gkh)^{-1/2}\varphi(gkh) \chi_\gamma(h) \de \mu_\gamma(h) \nonumber \\
&=\int\limits_{P(\gamma)} \rho_\gamma(gh)^{-1/2}\varphi(gh) \chi_\gamma(k^{-1}h) \de \mu_\gamma(h) \nonumber \\
&=\chi_\gamma(k)^{-1}(R_\gamma \varphi)(g),\label{invar}
\end{align}  
hence indeed $R_\gamma \varphi\in \Dc(G,P(\gamma), \chi_\gamma)$. 
The function
$$ \Gamma\times G \ni (\gamma, g) \mapsto (R_\gamma \varphi)(g)\in \CC $$
is continuous by Lemma~\ref{Wi07_LH.9} \eqref{Wi07_LH.9_item4}.
Moreover, the set $\{R_\gamma\varphi\mid \varphi\in \Cc_c(G)\}$ is dense in~$\Hc_\gamma$.
(See the proof of \cite[Th. C.33]{RW98}.)

\begin{remark}\label{e-contfield}
	\normalfont 
	For a family of Hilbert spaces $(\Hc_\gamma)_{\gamma\in \Gamma}$, 
	where $\Gamma$ is a topological space, 
	we will use the notion of continuity structure 
	$F\subseteq\prod\limits_{\gamma\in\Gamma}\Hc_\gamma$ 
	as in \cite[I.1]{Fe61}. 
	By definition, $F$ is a complex linear subspace with the properties that the function $\gamma\mapsto \Vert v_\gamma\Vert_\gamma$ is continuous for every section $v=(v_\gamma)_{\gamma\in\Gamma}\in F$ 
	and moreover the subset $\{v_\gamma\mid v\in F\}\subseteq\Hc_\gamma$ is dense for every $\gamma\in \Gamma$. 
	That notion is equivalently described in \cite[Prop.~10.2.3]{Di64}, 
	which ensures that there is a space  $\widetilde{F}$ of sections of 
	$(\Hc_\gamma)_{\gamma\in \Gamma}$ such that ${\Hb}_\Gamma = ((\Hc_\gamma)_{\gamma \in \Gamma}, \widetilde{F})$ is a continuous field of Hilbert 
	spaces and $F\subseteq\widetilde{F}$. 
	Specifically, $\widetilde{F}$ is the space of sections that can be approximated by sections in $F$, 
	uniformly on neighbourhoods of points in $\Gamma$.
	The space $\Gamma$ is paracompact and has finite covering dimension. 
	It follows by \cite[Lemma~10.8.7]{Di64} that 
	$\Hb_\Gamma$ is a  trivial continuous field of Hilbert spaces, 
	hence its corresponding continuous field of elementary $C^*$-algebras  $\Ac(\Hb_\Gamma)$ is trivial. 
\end{remark}

\begin{proposition}\label{e-contfield-H}
	The set of sections
	$$	
	F:=\{ (R_\gamma \varphi)_{\gamma\in\Gamma}\mid \varphi \in \Cc_c(G)\}\subseteq\prod_{\gamma\in\Gamma}\Hc_\gamma$$
	defines a continuity structure on the field of Hilbert spaces $(\Hc_\gamma)_{\gamma\in \Gamma}$.
\end{proposition}

\begin{proof} 
	By the discussion above, it suffices to prove that for every $\varphi\in \Cc_c(G)$ the function
	$$ \gamma \mapsto \norm{R_\gamma\varphi}_\gamma$$
	is continuous on $\Gamma$.
	
	To this end let $K\subseteq G$ be a compact subset with $\supp\varphi\subseteq K$. 
	Then there is a compact subset $K_\gamma\subseteq G/P(\gamma)$ with  
	$q_\gamma(\supp (R_\gamma\varphi))\subseteq K_\gamma$, and we  have
	\allowdisplaybreaks
	 \begin{align*}
		\norm{R_\gamma\varphi}^2_\gamma 
		= &
		\int\limits_{K_\gamma}  \int\limits_{P(\gamma)} 
		\rho_\gamma(gh)^{-1/2}
		\varphi(g h) \chi_\gamma (h) \overline{(R_\gamma\varphi)(g)} \de \mu_\gamma (h) 
		\de \nu_\gamma({\dot g})\\
		 \mathop{=}\limits^{(\dagger)} & \int\limits_{K_\gamma} \int\limits_{P(\gamma)}
		\rho_\gamma(gh))^{-1/2} \varphi(g h)  \overline{(R_\gamma\varphi)(gh)} \de \mu_\gamma (h) 
		\de \nu_\gamma({\dot g})
		\\
		 \mathop{=}\limits^{(\ddagger)}& \int\limits_{G} \rho_\gamma(g)^{1/2} \varphi(g)  \overline{(R_\gamma\varphi)(g)}\de \mu_G(g), 
		\end{align*}
where we have used \eqref{invar} to get $(\dagger)$, and $(\ddagger)$ follows from \eqref{rho_eq0}.
The required continuity is then a consequence of the fact that the functions
$(\gamma, g) \mapsto \rho_\gamma(g)^{1/2}$ and $(\gamma, g) \mapsto (R_\gamma \varphi)(g)$ are continuous, while $\varphi\in \Cc_c(G)$.
\end{proof}

\begin{lemma}\label{lemma_cont-rep}
	For every $f\in C^*(G)$,
	 $\phi, \psi\in \Cc_c(G)$ the function 
	$$ \gamma \mapsto \scalar{\pi_\gamma(f)R_\gamma\varphi}{R_\gamma\psi}_\gamma$$
	is continuous on $\Gamma$.
\end{lemma}

\begin{proof} 
Recall that the induced representation $\pi_\gamma=\Ind_{P(\gamma)}^G \chi_\gamma$ is the unitary representation of $G$ on  $\Hc_\gamma$ given by
	$$(\pi_\gamma (g) \phi)(s) = \big( \frac{\rho_\gamma(g^{-1}s)}{\rho_\gamma(s)}\big)^{1/2}\phi(g^{-1} s)$$ 
	for every $\phi\in \Dc(G,P(\gamma), \chi_\gamma)$.
	(See \cite[Thm.~C.33]{RW98}.)
	Then for every $\varphi\in \Cc_c(G)$, $s, t\in G$ we have
	\begin{align*}
	(\pi_\gamma(t) R_\gamma\varphi)(s)& = \big( \frac{\rho_\gamma(t^{-1}s)}{\rho_\gamma(s)}\big)^{1/2} (R_\gamma\varphi)(t^{-1}s)\\
	& = \big( \frac{\rho_\gamma(t^{-1}s)}{\rho_\gamma(s)}\big)^{1/2} \int\limits_{P(\gamma)}\rho_\gamma(t^{-1} s h)^{-1/2} \varphi(t^{-1} s h) \chi_\gamma(h) \de \mu_K(h)\\
	& =  \int\limits_{P(\gamma)}\rho_\gamma(s h)^{-1/2} \varphi(t^{-1} s h) \chi_\gamma(h) \de \mu_K(h), 
\end{align*}
where we have used that $\rho_\gamma(gh) =\rho_\gamma(g)\rho_\gamma(h)$ when $h \in P(\gamma)$ and $g\in G$. 

Consider first $f \in \Cc_c(G)$. 
For every $s\in G$ we  have 
	\allowdisplaybreaks
	\begin{align*}
		(\pi_\gamma(f) R_\gamma\varphi)(s) & = \int\limits_G f(t) (\pi_\gamma(t) R_\gamma\varphi)(s)  \de\mu_G(t) \\
&=
		\int\limits_{G} \int\limits_{P(\gamma)} 
		f(t) 
		\rho_\gamma(s h)^{-1/2}\varphi(t^{-1} sh) \chi_\gamma(h)\de \mu_\gamma(h) \de\mu_G(t).
	\end{align*}
	It follows that 
	\begin{align*}
	&	\scalar{\pi_\gamma(f)R_\gamma\varphi}{R_\gamma\psi}_\gamma\\
 & =\int\limits_{G/P(\gamma)}
	\int\limits_{G} \int\limits_{P(\gamma)} 
		f(t) 
		\rho_\gamma(s h)^{-1/2}\varphi(t^{-1} sh) \chi_\gamma(h)
	 \overline{(R_\gamma\psi)(s)} \de \mu_\gamma(h) \de\mu_G(t) \de \nu_\gamma({\dot s}).
	\end{align*}
	By \eqref{invar}, 
	$ \chi_\gamma(h)
	 \overline{(R_\gamma\psi)(s)}= 
	 \overline{(R_\gamma\psi)(sh)}
	 $.
	Since $f,\varphi,\psi$ have compact support  in $G$, and $R_\gamma\psi$ has compact support in $G$ modulo $P(\gamma)$ we can write
	\allowdisplaybreaks
	\begin{align*}
		& \scalar{\pi_\gamma(f)R_\gamma\varphi}{R_\gamma\psi}_\gamma\\
		& = \int\limits_{G}  \int\limits_{G/P(\gamma)}
	 \int\limits_{P(\gamma)} 
		f(t) 
		\rho_\gamma( s h)^{-1/2}\varphi(t^{-1} sh) 
	 \overline{(R_\gamma\psi)(sh)} \de \mu_\gamma(h)  \de \nu_\gamma({\dot s}) \de\mu_G(t)
	 \\
	 & 
	 = \int\limits_{G}  \int\limits_{G} f(t) 
		\rho_\gamma(s)^{1/2}  \varphi(t^{-1} s) 
	 \overline{(R_\gamma\psi)(s)}   \de\mu_G(s)  \de\mu_G(t)
	 \end{align*}
		Then, since $f$, $\varphi$ are continuous and have compact support and the functions
		$ (\gamma, s) \mapsto \overline{(R_\gamma\psi)(s)}$ and $(\gamma, s) \mapsto \rho_\gamma(s)$ 
		are continuous, we get the continuity in the statement for $f\in \Cc_c(G)$. 
		
		Assume now that $f\in C^*(G)$. 
		Then for every $\epsilon >0$ there is 
		$f_\epsilon \in \Cc_c(G)$ such that $\Vert f -f_\epsilon\Vert < \epsilon$. 
		Fix $\gamma_0\in \Gamma$. 
		Then,
		since the function
		$\gamma \to \Vert R_\gamma\phi\Vert_\gamma$ is continuous on $\Gamma$ for every $\phi \in \Cc_0(G)$, 
 we have
		$$ \vert \scalar{\pi_\gamma(f)R_\gamma\varphi}{R_\gamma\psi}_\gamma-
		\scalar{\pi_\gamma(f_\epsilon)R_\gamma\varphi}{R_\gamma\psi}_\gamma\vert 
		\le \epsilon \Vert {R_\gamma\varphi}\Vert_\gamma\Vert  {R_\gamma\psi}\Vert_\gamma \le C_\Omega \epsilon, $$
	 for $\gamma$ in a compact neighbourhood  
		$\Omega$ 
		of $\gamma_0$, 	where $C_\Omega$ is a constant that depends only on $\Omega$, $\varphi$ and $\psi$.
Therefore, the function $\gamma\mapsto  \scalar{\pi_\gamma(f)R_\gamma\varphi}{R_\gamma\psi}_\gamma$ can be approximated by continuous function on compact neighbourhoods of $\gamma_0$, hence it is continuous at $\gamma_0$. 
Since $\gamma_0$ is arbitrary we get the continuity in the statement. 
\end{proof}

\begin{theorem}\label{e-cont-field}
	In the Hypothesis~\ref{e-cont-hyp} above, if $\Sc\in \Subquot^{\Tr}(C^*(G))$ satisfies $\Gamma=\widehat{\Sc}$, 
	then there is a $*$-isomorphism 
	$$\Sc\mathop{\longrightarrow}\limits^{\sim} \Ac(\Hb_\Gamma), \quad 
	f\mapsto (\pi_\gamma(f))_{\gamma\in\Gamma}$$
	and the $C^*$-algebra $\Sc$ is Morita equivalent to a commutative $C^*$-algebra. 
	\end{theorem}

\begin{proof}
	By hypothesis, every compact subset of $\widehat{\Sc}$ has finite covering dimension,
	hence $J_\Sc=K_\Sc$ by Proposition~\ref{PI}\eqref{PI_item3}. 
	Therefore, by Corollary~\ref{isom}, it suffices to prove that for arbitrary $f\in K_\Sc^+$ we have $(\pi_\gamma(f))_{\gamma\in\Gamma}\in \Ac(\Hb_\Gamma)$.  
	
	Let $\Jc_1\subseteq\Jc_2$ be closed two-sided ideals of $C^*(G)$ with $\Sc=\Jc_2/\Jc_1$. 
	For every $ f\in J_{\Sc}=K_{\Sc}$, 
	there is $f_0\in K_{\Jc_2}$ such that $f = f_0 + \Jc_2$, by Proposition~\ref{PI}, since $\Gamma$
	has finite covering dimension. 
    Since $\pi_\gamma(J_2) =0$ for every $\gamma\in \Gamma$, 
it follows by Lemma~\ref{lemma_cont-rep} that, for every $\varphi, \psi\in \Cc_c(G)$,  the function
	$$ \gamma\mapsto \dual{\pi_\gamma(f)R_\gamma \varphi}{R_\gamma\psi}= \dual{\pi_\gamma(f_0)R_\gamma \varphi}{R_\gamma\psi}
	$$
	is continuous on $\Gamma$. 
	
	On the other hand, as $f\in J_{\Sc}$, the function 
	$$ \gamma \mapsto \Tr(\pi_\gamma(f))$$ 
	is continuous on $\Gamma$, and has compact support. 
	
	By \cite[Cor., page 260]{Fe61} 
	 we obtain 
	$(\pi_\gamma(f))_{\gamma\in\Gamma}\in \Ac(\Hb_\Gamma)$, 
	which completes the proof. 
\end{proof}

\subsection{The case of exponential Lie groups}
In this subsection 
we denote by $G$ an exponential Lie group with its Lie algebra~$\gg$.  

We first recall the definition of the Bernat-Kirillov correspondence.
\begin{definition}\label{polariz_def}
	\normalfont 
	For any $\xi\in\gg^*$ we define 
	$$\begin{aligned}
		\Sg(\xi):=
		&\{\pg\in\Gr_{\alg}(\gg)\mid [\pg,\pg]\subseteq\Ker\xi\}, \\
		\Pg(\xi):=
		&\{\pg\in\Sg(\xi)\mid 2\dim\pg=\dim\gg+\dim\gg(\xi)\}.\\
		\Ig(\xi):=
		&\{\pg\in\Sg(\xi)\mid [\Ind_P^G(\chi_\xi\vert_P)]\in\widehat{G}.\}
	\end{aligned}$$
	The elements of $\Pg(\xi)$ are called \emph{polarizations} at $\xi\in\gg^*$. 
	Here we used the notation 
	$$\chi_{\xi}\colon G\to\TT,\quad \chi_\xi(x):=\ee^{\ie\langle\xi,x\rangle}$$
	hence $\chi_\xi\vert_P\colon P\to\TT$ is a continuous morphism of Lie groups  for every $\pg\in\Sg(\xi)$ with its corresponding connected subgroup $P\subseteq G$. 
\end{definition}

\begin{remark}
	\label{BK_rem}
	\normalfont 
	In the setting of Definition~\ref{polariz_def}, 
	the following assertions hold: 
	\begin{enumerate}[{\rm(i)}]
		\item 
		For every $\xi\in\gg^*$ one has  
		$\Ig(\xi)=\{\pg\in\Pg(\xi)\mid P.\xi=\xi+\pg^\perp\}\ne\emptyset$ 
		and $\Ig(\xi)\subseteq\Pg(\xi)\subseteq\Sg(\xi)$.
		\item If $\xi\in\gg^*$, 
		then $[\Ind_{P_1}^G(\chi_\xi\vert_{P_1})]=
		[\Ind_{P_2}^G(\chi_\xi\vert_{P_2})]\in\widehat{G}$ 
		for all $\pg_1,\pg_2\in\Ig(\xi)$. 
		\item If $\xi_j\in\gg^*$ and $\pg_j\in\Ig(\xi_j)$ for $j=1,2$, 
		then $[\Ind_{P_1}^G(\chi_{\xi_1}\vert_{P_1})]=
		[\Ind_{P_2}^G(\chi_{\xi_2}\vert_{P_2})]$ 
		if and only if there exists $x\in G$ with $\xi_2=x.\xi_1$. 
	\end{enumerate}
	See \cite[Ch. VI, \S\S 2--3]{BCDLRRV72} for proofs of these assertions. 
	
	As proved in 
{\cite[Ch. VI, \S 2]{BCDLRRV72}},
	the Bernat-Kirillov correspondence 
	$$\kappa_G\colon \gg^*/G\to\widehat{G}, \; \; 
	\kappa_G(G.\xi)=[\Ind_P^G(\chi_\xi\vert_P)]\,\, \text
	{for all $\xi\in\gg^*$ and $\pg\in\Ig(\xi)$} $$
	is well defined and bijective. 
\end{remark}

For the exponential group $G$ above, 
we let  $q\colon \gg^* \to \gg^*/G$ be the quotient map defined by the coadjoint action of~$G$ and denote simply by  $\kappa:=\kappa_G \colon\gg^*/G\to\widehat{G}$
the Kirillov-Bernat bijection.

We assume that $G$ satisfies the following hypothesis. 
\begin{hypothesis}\label{contpolar-hyp}
\normalfont 
	\begin{enumerate}[{\rm(i)}]
		\item\label{contpolar-hyp_item1}   
		$\Gamma\in \LC^{\Tr}(\widehat{G})$ and  $(q\circ\kappa)^{-1}(\Gamma)\subseteq\gg^*\setminus[\gg,\gg]^\perp$. 
\item\label{contpolar-hyp_item2}
		$\sigma\colon \Gamma\to q^{-1}(\Gamma) $ 
		is a continuous mapping with  
		$q(\sigma(\ell)) =\ell$ for all $\ell \in q^{-1}(\Gamma)$. 
		\item\label{contpolar-hyp_item3}   
		$ \pg\colon \Gamma \to \Gr(\gg)$ is a continuous mapping such that 
		$\pg(\gamma)\in\Pg(\sigma(\gamma)) $ and the real polarization $\pg(\gamma)$ satisfies Puk\'anszky's condition $P(\gamma)\sigma(\gamma)=\sigma(\gamma)+\pg(\gamma)^\perp$ 
		for every $\gamma\in\Gamma$, where $P(\gamma)$ is the connected closed subgroup of $G$ that corresponds to $\pg(\gamma)$
\end{enumerate}  
\end{hypothesis}

\begin{lemma}\label{equiv}
If the exponential Lie group $G$ and the subset set $\Gamma\subseteq\widehat{G}$ satisfy  Hypothesis~\ref{contpolar-hyp} above, then 
they satisfy Hypothesis~\ref{e-cont-hyp} as well.
\end{lemma}

\begin{proof}
In Hypothesis~\ref{contpolar-hyp}, the topological space $\Gamma$ is paracompact and its covering dimension is finite. 
Moreover, the covering dimension of every compact subset of $\Gamma$ is finite. 

In fact, $\Gamma$ is the spectrum of a continuous-trace separable $C^*$-algebra, 
and paracompact, as above. 
Moreover, the mapping $\sigma\colon\Gamma\to q^{-1}(\Gamma)$ is a homeomorphism (with its inverse $q\vert_{q^{-1}(\Gamma)}$), 
hence it suffices to show that $q^{-1}(\Gamma)$, regarded as a topological subspace of $\gg^*$, has finite covering dimension. 
We have seen above that $\Gamma$ is a locally compact Hausdorff space,  
hence $q^{-1}(\Gamma)$ is in turn a locally compact Hausdorff space, 
and then it is a locally closed subset of $\gg^*$ by \cite[Lemma 1.26]{Wi07}. 
That is,  
$$q^{-1}(\Gamma)=E\cap D$$
for a suitable closed subset $E\subseteq\gg^*$ 
and a suitable open subset $D\subseteq \gg^*$. 
The covering dimension of the closed subset $E$ is finite since it 
is less than the covering dimension of $\gg^*$, cf. \cite[Ch. 3, Prop. 1.5 and Th. 2.7]{Pea75}. 
Moreover, since we have seen above that $\Gamma$ is a metrizable space, it follows that 
both $q^{-1}(\Gamma)$ and its subspace $E$ are metrizable. 
Then $D\cap E$ is an open subset of the metrizable space $E$, hence $D\cap E$ is an $F_\sigma$-subset of $E$. 
The space $E$ is metrizable, hence is a normal space, and then the covering dimension of its $F_\sigma$-subset $D\cap E$ is less than the covering dimension of $E$ by \cite[Ch. 3, Cor. 6.3]{Pea75}. 
Consequently, the covering dimension of $E\cap D$ is finite. 
That is, $q^{-1}(\Gamma)$ has finite covering dimension. 

Moreover, since $\Gamma$ is a Hausdorff space, every compact subset of $\Gamma$ is closed, 
hence its covering dimension is less than the covering dimension of $\Gamma$ by \cite[Ch. 3, Prop. 1.5]{Pea75} again. 

By Remark~\ref{eqtop} and Hypothesis~\ref{contpolar-hyp} we have that
the mapping 
$$ \Gamma \to \Sigma(G), \quad \gamma\mapsto P(\gamma)$$
is continuous where $\Sigma(G)$ is endowed with the Chabauty-Fell topology.

Let 
$$\chi_\gamma\colon G\to\TT, 
\quad \chi_\gamma(g) = \ee^{\ie\dual{\sigma(\gamma)}{\log_G g}},$$ 
and denote by 
$\chi_\gamma$ also its restriction to $P(\gamma)$, that is, the character of $P(\gamma)$. 
so that $\chi_\gamma\vert_{P(\gamma)}\in\Hom(P(\gamma),\TT)$. 
Then $\chi_\gamma$ satisfies \eqref{cont-hyp_item2} in Hypothesis~\ref{e-cont-hyp}.

Finally, the representation $\pi_\gamma:=\Ind_{P(\gamma)}^G(\chi_\gamma\vert_{P(\gamma)})\colon G\to\Bc(\Hc_\gamma)$
is irreducible since the real polarization $\pg(\gamma)$ satisfies Puk\'anszky's condition, 
cf. \cite[Ch. VI, Prop. 3.2]{BCDLRRV72}.
\end{proof}

\begin{corollary}\label{cont-field}
	Assume the exponential Lie group $G$ and the subset $\Gamma\subseteq\widehat{G}$ satisfy
	the Hypothesis~\ref{contpolar-hyp} above. 
	If $\Sc\in \Subquot^{\Tr}(C^*(G))$ satisfies $\Gamma=\widehat{\Sc}$, 
	then we have the $*$-isomorphism 
	$$\Sc\mathop{\longrightarrow}\limits^{\sim} \Ac(\Hb_\Gamma), \quad 
	f\mapsto (\pi_\gamma(f))_{\gamma\in\Gamma}$$
	and the $C^*$-algebra $\Sc$ is Morita equivalent to a commutative $C^*$-algebra. 
	\end{corollary}
\begin{proof}
This is a direct consequence of Lemma~\ref{equiv} and Theorem~\ref{e-cont-field}.
\end{proof}

\section{Main results}\label{section4}

In order to state our main results 
(Theorem~\ref{finetriv} and Corollary~\ref{solv})
we need first to recall the setting of \cite[Sect. 6]{BB24}.

We start with a simple lemma.
\begin{lemma}
	\label{complex1}
	Let $\Vc$ be an $\RR$-linear space with its complexification $\Vc_\CC$ and define the antilinear map $C\colon\Vc_\CC\to\Vc_\CC$, $C(x+\ie y):=x-\ie y$. 
	Then the following assertions hold.
	\begin{enumerate}[{\rm(i)}]
		\item\label{complex1_item1} 
		For every $\CC$-linear subspace $\Wc\subseteq\Vc$ satisfying $C(\Wc)=\Wc$, 
		if we denote $\Xc:=\Wc\cap\Vc$ then $\Wc=\Xc_\CC=\Xc+\ie\Xc$. 
		\item\label{complex1_item2} 
		If $\Xc,\Yc\subseteq \Vc$ are $\RR$-linear subspaces, then 
		\begin{align}
			\label{complex1_eq1}
			& \Xc\subsetneqq\Yc\iff \Xc_\CC\subsetneqq\Yc_\CC; \\
			\label{complex1_eq2}
			& (\Xc\cap\Yc)_\CC=\Xc_\CC \cap\Yc_\CC;  \\
			& (\Xc+\Yc)_\CC=\Xc_\CC + \Yc_\CC. \nonumber
		\end{align}
	\end{enumerate}
\end{lemma}

\begin{proof}
	Straightforward. 
\end{proof}

\begin{definition}
	\label{complex2}
	\normalfont 
	Let $\KK\in\{\RR,\CC\}$. 
	If $\Vc$ is a $\KK$-linear space  
	and $B\colon \Vc\times\Vc\to\KK$ is a skew-symmetric $\KK$-bilinear functional, we denote $$N(B):=\Vc^{\perp_B}:=\{v_0\in\Vc\mid (\forall v\in\Vc)\quad B(v_0,v)=\{0\}\}.$$
	We now assume $m:=\dim_\KK\Vc<\infty$ and we fix a sequence of linear subspaces  
	$$\{0\}=\Vc_0\subsetneqq\Vc_1\subsetneqq\cdots\subsetneqq\Vc_m=\Vc$$
	with $\dim_\KK\Vc_j=j$ for $j=0,\dots,m$,  
	set $B_j:=B\vert_{\Vc_j\times\Vc_j}$ for $j=0,\dots,m$, 
	and define 
	\begin{equation}\label{pgB}
		\pg(B):=N(B_1)+\cdots+N(B_m)\in\Gr(\Vc).
	\end{equation}
	For every $\KK$-linear subspace $\Wc\subseteq \Vc$ we define 
	$$\jump\Wc:=
	\{j\in\{1,\dots,m\}\mid \Vc_j\not\subset\Vc_{j-1}+\Wc\}.$$
	We now set $\pg^0(B):=\Vc$. 
	Inductively, assume $k\ge 0$ is an integer and we have already defined the linear subspaces $\pg^0(B)\supsetneqq\cdots\supsetneqq\pg^k(B)$ of~$\Vc$. 
	If the condition $\pg^k(B)\not\perp_B\pg^k(B)$ is satisfied, then we define 
	\begin{align}
		\label{D1_eq1}
		i_{k+1}(B):=
		& \min\{i\in\{0,\dots,m\}\mid \Vc_i\cap\pg^k(B)\not\perp_B\pg^k(B)\},\\ 
		\label{D1_eq2}
		\pg^{k+1}(B):=
		&(\Vc_{i_{k+1}(B)}\cap\pg^k(B))^{\perp_B}\cap\pg^k(B)
	\end{align}
	Moreover, we define 
	\begin{equation}
		\label{D1_eq3}
		j_{k+1}(B):=\min\{j\in\{0,\dots,m\}\mid \Vc_j\cap\pg^k(B)\not\subset\pg^{k+1}(B)\}.
	\end{equation}
	Denoting $d:=\dim_\KK(\Vc/N(B))$ we have $d=2\dim_\KK(\Vc/\pg(B))$ and  
	$$\Vc=\pg^0(B)\supsetneqq\cdots\supsetneqq\pg^{d/2}(B)=\pg(B),$$
	cf., \cite[Lemmas 5.3 and 5.5]{BB24}.
\end{definition}

\begin{remark}\label{remark4.3}
	\normalfont 
	In Definition~\ref{complex2} we have 
	$$i_k(B),j_k(B)\in\jump N(B), \quad i_k(B)<j_k(B), \quad i_k(B)<i_{k+1}(B)$$
	cf. \cite[Lemmas 5.6]{BB24}.
	Moreover, the mapping 
	$$\{1,\dots,d\}\to\jump N(B)\setminus\jump \pg(B), \quad k\mapsto i_k(B) $$
	is a well-defined increasing bijection by \cite[Lemmas 5.7]{BB24}, 
	and 
	the mapping 
	$${\mathbf j}(B)\colon \{1,\dots,d\}\to\jump \pg(B), \quad k\mapsto j_k(B)$$ 
	is a well-defined bijection, cf., \cite[Lemmas 5.8]{BB24}.
\end{remark}

\begin{lemma}
	\label{complex3}
	Let $\Vc$ be an $\RR$-linear space with a family of $\RR$-linear subspaces 
	$$\{0\}=\Vc_0\subsetneqq \Vc_1\subsetneqq\cdots\subsetneqq\Vc_m=\Vc$$
	with $\dim\Vc_j=j$ for $j=1,\dots,m$, 
	with their corresponding complexified spaces 
	$$\{0\}=\Vc_{0,\CC}\subsetneqq \Vc_{1,\CC}\subsetneqq\cdots\subsetneqq\Vc_{m,\CC}=\Vc_\CC$$
	If $B\colon \Vc\times\Vc\to\RR$ is a skew-symmetric $\RR$-bilinear functional, 
	we define the skew-symmetric $\RR$-bilinear functionals $B_j:=B\vert_{\Vc_j\times\Vc_j}\colon \Vc_j\times\Vc_j\to\RR$ for $j=0,\dots,m$, 
	with their corresponding 
	skew-symmetric 
	$\CC$-bilinear extensions $B_\CC\colon\Vc_\CC\times\Vc_\CC\to\CC$ 
	and 
	$B_{j,\CC}:=B_\CC\vert_{\Vc_{j,\CC}\times\Vc_{j,\CC}}\colon \Vc_{j,\CC}\times\Vc_{j,\CC}\to\CC$ for $j=0,\dots,m$. 	
	Then the following assertions hold: 
	\begin{enumerate}[{\rm(i)}]
		\item	\label{complex3_item1} 
		We have $N(B_{j,\CC})=N(B_j)_\CC$ for $j=1,\dots,m$. 
		\item\label{complex3_item2} 
		We have $\pg^k(B_\CC)=\pg^k(B)_\CC$, $i_k(B_\CC)=i_k(B)$, and $j_k(B_\CC)=j_k(B)$ for $k=1,\dots,d/2$. 
		\item\label{complex3_item3} 
		For every $\RR$-linear subspace $\Wc\subseteq\Vc$ we have $\jump\Wc=\jump\Wc_\CC$. 
	\end{enumerate}
\end{lemma}

\begin{proof}
	Straightforward, using Lemma~\ref{complex1}. 
\end{proof}

Let $G$ be a completely solvable Lie group with its Lie algebra $\gg$ and 
fix a Jordan-H\"older sequence, that is, a chain of ideals 
\begin{equation}
	\label{JH}
	\{0\}=\gg_0\subseteq\gg_1\subseteq\cdots\subseteq\gg_m=\gg
\end{equation}
with $\dim\gg_j=j$ for $j=0,1,\dots,m$. 
We denote by $\langle\cdot,\cdot\rangle\colon \gg^*\times\gg\to\RR$ and we define 
$$(\forall\xi\in\gg^*)\quad B_\xi\colon\gg\times\gg\to\RR,\quad B_\xi(x,y):=\langle\xi,[x,y]\rangle.$$

\begin{definition}
\normalfont 
We introduce  the sets 
\begin{equation}
	\label{Lcont_eq1}
	J_m:=\{\mathbf{k}=(k_1,\dots,k_m)\in\NN^m\mid 0\le k_j\le j\text{ for } j=0,\dots,m\}
\end{equation}	
and 
\begin{equation}
	\label{Xiks}
	\Xi_{\mathbf{k}}:=\{\xi\in\gg^*\mid \dim\gg_j(\xi\vert_{\gg_j})=k_j\text{ for }j=1,\dots,m\}
\end{equation}
for $\mathbf{k}=(k_1,\dots,k_m)\in J_m$, 
where $\gg_j(\xi\vert_{\gg_j})=\{x\in\gg_j\mid(\forall y\in\gg_j)\ \langle\xi,[x,y]=0\}$
for $j=1,\dots,m$. 
It is clear that we have the partition 
\begin{equation}
\label{fine}
\gg^*=\bigsqcup_{\mathbf{k}\in J_m}\Xi_{\mathbf{k}}.
\end{equation}
where some of the sets in the right-hand side may be empty. 
This 
partition is called the \emph{fine layering} of~$\gg^*$ and 
the $G$-invariant sets $\Xi_{\mathbf{k}}$ are called the \emph{fine layers of $\gg^*$} with respect to the Jordan-H\"older sequence~\eqref{JH}. 
\end{definition}

For arbitrary $\xi\in\gg^*$, if we set
$$\pg_\alg(\xi):=\pg(B_\xi)$$
we have $\gg(\xi)\subseteq \pg_\alg(\xi)$, therefore $\jump \pg_\alg(\xi)\subseteq\jump \gg(\xi)$. 
If we denote $d:=\frac{1}{2}\dim(\gg/\gg(\xi))=\frac{1}{2}\dim (G\xi)\in\NN$, then, by Remark~\ref{remark4.3}, the elements of $\jump \pg_\alg(\xi)$ can be labeled via the bijective map 
$$\mathbf{j}(B) \colon \{1,\dots,d\}\to\jump \pg_\alg(\xi),\quad k\mapsto j_k(B).$$

The multi-index $\mathbf{k}=(k_1,\dots,k_m)\in J_m$ with $\dim\gg_j(\xi\vert_{\gg_r})=k_r$ for $r=1,\dots,m$ is uniquely determined by the pair~$(e,\mathbf{j})$, where $e=\jump\gg(\xi)\subseteq \{1,\dots,m\}$ with $\card e=2d$ and $\mathbf{j}=\mathbf{j}(B) \colon\{1,\dots,d\}\to e$ is the injective map above,
cf. \cite[Lemma 5.9]{BB24}.

\begin{definition}
	\label{complex4}
	\normalfont 
	Let $\lambda_1,\dots\lambda_m\colon\gg\to\RR$ be the roots with respect to the Jordan-H\"older series~\eqref{JH}. 
	Using the notation in Definition~\ref{complex2} for $\Vc=\gg$ and $\Vc_j=\gg_j$ for $j=1,\dots,m$, 
	we define for arbitrary $\xi\in\gg^*$
	\begin{equation*}
		(\forall\xi\in\gg^*)\quad 	\bt(\xi):=\{k\in\{1,\dots,d/2\}\mid 
		\pg^{k-1}(B_{\xi})\cap\Ker\lambda_{i_k}=\pg^k(B_{\xi})\}.
	\end{equation*}
Thus, if we denote by $J$ the set of all triples 
$(e,\mathbf{j}, \bt)$, where $e\subseteq \{1,\dots,m\}$, $\card e$ is an even integer, 
$\mathbf{j}\colon\{1,\dots,(\card e)/2\}\to e$ is an injective map,  and
$\bt \subseteq \{1, \dots, (\card e)/2\}$, 
then we obtain from  \eqref{fine} 
the 
\emph{ultrafine layering}
\begin{equation}
	\label{fine1}
	\gg^*=\bigsqcup_{(e,\mathbf{j}, \bt)\in J}\Xi_{(e,\mathbf{j}, \bt)}
\end{equation}
where, again, some sets in the right-hand side may be empty. 
Here 
$$\Xi_{(e,\mathbf{j}, \bt)} := \{\xi \in \Xi_{\mathbf k} \mid \bt(\xi) = \bt\}, $$
where ${\mathbf k}\in J_m$ corresponds to the pair $(e, {\mathbf j})$ as explained 
 above. 
\end{definition}

\begin{remark}
\normalfont 
The layers $\Xi_{(e,\mathbf{j}, \bt)}$ in \eqref{fine1} are $G$-invariant subsets of~$\gg^*$. 
Therefore, we can use 
the quotient map $q\colon \gg^*\to\gg^*/G$ corresponding to the coadjoint action and the Kirillov-Bernat bijection $\kappa\colon\gg^*/G\to\widehat{G}$, 
to
define the partition 
\begin{equation*}
	\widehat{G}=\bigsqcup_{(e, {\mathbf j}, \bt) \in J}
	\Gamma_{(e, \mathbf{j}, \bt) }.
\end{equation*}
where some of the sets 
$\Gamma_{(e, \mathbf{j}, \bt)}:=\kappa(q(\Xi_{(e, \mathbf{j}, \bt)}))$ in the right-hand side may be empty. 
This partition is called the \emph{ultrafine layering} of 
the unitary dual~$\widehat{G}$ and the sets $\Gamma_{\mathbf{k}}$ are called the \emph{ultrafine layers of $\widehat{G}$} with respect to the Jordan-H\"older sequence~\eqref{JH}. 
\end{remark}

The ultrafine layering is built directly from the Jordan-H\"older sequence~\eqref{JH}, which in turn exists since the Lie algebra $\gg$ is completely solvable. 
For more general solvable Lie algebras, such a construction is not so direct as it involves passing to complexifications in order to accommodate the complex-valued roots of the Lie algebra under consideration. 
	In order to be able to use results from the earlier literature, we show  below that these two approaches lead to the same result in the present framework of completely solvable Lie algebras.

\begin{lemma}\label{lemma4.10}
	Let $\lambda_1,\dots\lambda_m\colon\gg\to\RR$ be the roots with respect to the Jordan-H\"older series~\eqref{JH}. 
	For every $\xi\in\gg^*$ we denote by $\xi_\CC\colon\gg_\CC\to\CC$ its $\CC$-linear extension with the corresponding $\CC$-bilinear functional $B_{\xi_\CC}\colon\gg_\CC\times\gg_\CC\to\CC$, $B_{\xi_\CC}(v,w):=\xi_\CC([v,w])$. 
	Using the notation in Definition~\ref{complex2} for $\Vc=\gg$ and $\Vc_j=\gg_j$ for $j=1,\dots,m$, 
	we have 
	\begin{equation*}
		\bt(\xi)
		=
		\{k\in\{1,\dots,d/2\}\mid 
		\pg^{k-1}(B_{\xi_\CC})\cap\Ker\lambda_{i_k,\CC}=\pg^k(B_{\xi_\CC})\}
	\end{equation*}
for every $\xi\in\gg^*$.
\end{lemma}

\begin{proof}
For arbitrary $k\in\{1,\dots,d/2\}$ 
we have 
$$\pg^{k-1}(B_\xi)_\CC=\pg^{k-1}(B_{\xi_\CC})
\text{ and }
\pg^k(B_\xi)_\CC=\pg^k(B_{\xi_\CC})$$
by  Lemma \ref{complex3}\eqref{complex3_item2}.  
On the other hand 
$(\Ker\lambda_{i_k})_\CC=\Ker\lambda_{i_k,\CC}$.  
Therefore, using equations \eqref{complex1_eq1}--\eqref{complex1_eq2} in Lemma~\ref{complex1}, 
we obtain
$$\pg^{k-1}(B_\xi)\cap\Ker\lambda_{i_k}=\pg^k(B_\xi)
\iff 
\pg^{k-1}(B_{\xi_\CC})\cap\Ker\lambda_{i_k,\CC}=\pg^k(B_{\xi_\CC})
$$ 
which implies the assertion.
\end{proof}

\begin{remark}\label{currey}
	\normalfont
 Lemma~\ref{lemma4.10} above shows that 
	the ultrafine layering 
	from \cite[Th. 2.8]{Cu92a} and and \cite[Thm.1.5 and 3.1]{Cu92b} 
	coincides with the one introduced in Definition~\ref{complex4} above.
\end{remark}

\begin{theorem}
	\label{finetriv}
	If $G$ is a completely solvable Lie group and $\Sc$ is a subquotient of $C^*(G)$ such that 
	$\widehat{\Sc}$ is included in one of the ultrafine layers 
	with respect to a Jordan-H\"older sequence, then 
	$\Sc\in \Subquot^{\Tr}(C^*(G))$ and 
	there exists a $*$-isomorphism $\Sc\simeq \Cc_0(\widehat{\Sc},\Kc(\Hc))$ 
	for a suitable separable Hilbert space~$\Hc$. 
	\end{theorem}

\begin{corollary}
	\label{solv}
	The $C^*$-algebra of every completely solvable Lie group is solvable. 
\end{corollary}

\begin{remark}
	\normalfont 
	In the special case of nilpotent Lie groups, the ultrafine layering coincides with the fine layering. 
	Therefore, in this case, the result of Corollary~\ref{solv} 
	coincides with the result obtained earlier earlier in \cite[Th. 4.11]{BBL16} with a completely different proof. 
\end{remark}

\begin{proof}[Proof of Theorem~\ref{finetriv}]
We denote $\Gamma:=\widehat{\Sc}$. 
Since $\Gamma$ is contained in a ultrafine layer, we have either $(q\circ\kappa)^{-1}(\Gamma)\subseteq[\gg,\gg]^\perp$ or $(q\circ\kappa)^{-1}(\Gamma)\subseteq\gg^*\setminus[\gg,\gg]^\perp$. 
In the first of these cases the subquotient $\Gamma$ is commutative, so the assertion follows at once. 

Now let us assume $(q\circ\kappa)^{-1}(\Gamma)\subseteq\gg^*\setminus[\gg,\gg]^\perp$. 
We check that all the conditions in Hypothesis~\ref{contpolar-hyp} are satisfied. 

We first 
recall from Remark~\ref{currey} that, 
by Lemma~\ref{lemma4.10}, 
the ultrafine layering 
from \cite[Th. 2.8]{Cu92a} and \cite[Thm.1.5 and 3.1]{Cu92b} coincides with the one introduced in Definition~\ref{complex4} 
above.

Condition~\eqref{contpolar-hyp_item1} follows from \cite[Thm.~3.1]{Cu92b}
and Remark~\ref{currey}.
Condition~\eqref{contpolar-hyp_item2} follows again by Remark~\ref{currey}
and \cite[Th. 2.8]{Cu92a} or \cite{CuPn89}, 
which shows that 
every ultrafine layer in $\gg^*$ has a cross-section.

For condition~\eqref{contpolar-hyp_item2} we note that $\pg_\alg
\colon\gg^*\to\Gr(\gg)$ is  the Vergne polarization mapping.
Thus for every $\xi$,  $\pg_{\alg}(\xi)$ is a real polarization at $\xi$  that satisfies Puk\'anszky's condition. 
(See \cite[Th. 4.3.6 and Cor. 4.3.7]{BCDLRRV72} or \cite[Cor. 3.2]{Grl92}.) 
On the other hand, the mapping $\xi \to \pg_\alg(\xi)$ is continuous on every fine layer $\Xi_{\mathbf{k}}\subseteq\gg^*$ by~\cite[Th. 6.5]{BB24}, hence also on $\Gamma$. 

Thus all the conditions in Hypothesis~\ref{contpolar-hyp} are satisfied, and then Corollary~\ref{cont-field} is applicable. 
Therefore $\Sc$ is $*$-isomorphic to the $C^*$-algebra of sections of a trivial continuous field of infinite-dimensional separable Hilbert spaces. 
This directly implies the assertion. 
\end{proof}

\begin{proof}[Proof of Corollary~\ref{solv}]
We fix a Jordan-H\"older sequence in $\gg$. 
By Remark~\ref{currey} and \cite[Th. 3.1]{Cu92b}, there exists a finite increasing family of  $G$-invariant dense open subsets 
$$\emptyset=\Omega_0\subseteq\Omega_1\subseteq\cdots\subseteq\Omega_N=\gg^*$$
such that 
 for every $r\in\{1,\dots,N\}$  there exists $(e,\mathbf{j}, \bt)\in J$ with $\Omega_r\setminus\Omega_{r-1}\subseteq\Xi_{(e,\mathbf{j}, \bt)}$. 

For $r=0,\dots,N$, let $\Jc_r\subseteq C^*(G) $ be the closed two-sided ideal corresponding to the open subset $\kappa(q(\Omega_r))\subseteq\widehat{G}$, 
that is, $\widehat{\Jc_r}=\kappa(q(\Omega_r))$. 
Then we have 
$$\{0\}=\Jc_0\subseteq\Jc_1\subseteq\cdots\subseteq\Jc_N=C^*(G).$$ 
Moreover, denoting $\Sc_r:=\Jc_r/\Jc_{r-1}$, we have $\widehat{\Sc_r}=\kappa(q(\Omega_r\setminus\Omega_{r-1}))$ is contained in an ultrafine layer.  
Therefore, by using Theorem~\ref{finetriv},  
we obtain  a $*$-isomorphism $\Sc_r\simeq \Cc_0(\widehat{\Sc},\Kc(\Hc_r))$ 
for a suitable separable Hilbert space~$\Hc_r$. 
This shows that the $C^*$-algebra $C^*(G)$ is solvable. 
\end{proof}

\subsection*{Acknowledgment}
We wish to thank the Referee for several pertinent remarks that improved the exposition of our paper.


\begin{thebibliography}{10000000}
	

	\bibitem[ArCu20]{ArCu20}
	\textsc{D.~Arnal, B.~Currey}, 
	\textit{Representations of solvable Lie groups}. 
	New Mathematical Monographs 39. Cambridge University Press, Cambridge, 2020.
	
	\bibitem[ArCu24]{ArCu24}
	\textsc{D.~Arnal, B.~Currey}, 
	Harmonic analysis on inhomogeneous nilpotent Lie groups. 
	\textit{J. Lie Theory} \textbf{34} (2024), no. 4, 873--910. 
	
	\bibitem[BB21]{BB21} 
	\textsc{I. Belti\c t\u a, D. Belti\c t\u a}, 
	On the isomorphism problem for $C^*$-algebras of nilpotent Lie groups. 
    \textit{J. Topol. Anal.} \textbf{13} (2021), no. 3, 753--782.
	
		\bibitem[BB24]{BB24} 
	\textsc{I. Belti\c t\u a, D. Belti\c t\u a}, 
	Continuous selection of Lagrangian subspaces. 
	\textit{Linear Multilinear Algebra} \textbf{72} (2024), no. 15, 2439--2465. 
	
	
	\bibitem[BB25]{BB24_RIMS} 
	\textsc{I. Belti\c t\u a, D. Belti\c t\u a}, 
	$C^*$-rigidity of the Heisenberg group.
	\textit{Publ. Res. Inst. Math. Sci.} (to appear).
	
		\bibitem[BBL16]{BBL16} 
	\textsc{I.~Belti\c t\u a, D.~Belti\c t\u a, J.~Ludwig}, 
	Fourier transforms of $C^*$-algebras of nilpotent Lie groups. 
	\textit{Int. Math. Res. Not. IMRN} \textbf{2017}, no. 3, 677--714.
	
	
	\bibitem[BC$^+$72]{BCDLRRV72}
	\textsc{P.~Bernat, N.~Conze, M.~Duflo, M.~L\'evy-Nahas, M.~Raïs, P.~Renouard,  M.~Vergne}, 
	\textit{Repr\'esentations des groupes de Lie r\'esolubles}.
	Monographies de la Soci\'et\'e Math\'ematique de France, No. 4. Dunod, Paris, 1972.
	
	
	
	\bibitem[Cu92a]{Cu92a}
	\textsc{B.N.~Currey}, 
	The structure of the space of coadjoint orbits of an exponential solvable Lie group. \textit{Trans. Amer. Math. Soc.} \textbf{332} (1992), no. 1, 241--269.
	
	\bibitem[Cu92b]{Cu92b}
	\textsc{B.N.~Currey}, 
	A continuous trace composition sequence for $C^*(G)$ where $G$ is an exponential solvable Lie group. 
	\textit{Math. Nachr.} \textit{159} (1992), 189--212.
	
	\bibitem[CuPn89]{CuPn89}
	\textsc{B.N.~Currey, R.C.~Penney}, 
	The structure of the space of co-adjoint orbits of a completely solvable Lie group. 
	\textit{Michigan Math. J.} \textbf{36} (1989), no. 2, 309--320. 
	
	\bibitem[Di64]{Di64}
	\textsc{J.~Dixmier},
	\textit{Les $C^{\ast} $-alg\`ebres et leurs repr\'esentations}. 
	Cahiers Scientifiques, Fasc. \textbf{XXIX} Gauthier-Villars \& Cie, \'Editeur-Imprimeur, Paris, 1964.
	

	\bibitem[Dy78]{Dy78}
	\textsc{A.~Dynin}, 
	Inversion problem for singular integral operators: $C^*$-approach. 
	\textit{Proc. Nat. Acad. Sci. U.S.A.} \textbf{75} (1978), no. 10, 4668--4670.
	
	

	
	\bibitem[Ech96]{Ech96}
	\textsc{S.~Echterhoff}, 
	Crossed products with continuous trace. 
	\textit{Mem. Amer. Math. Soc.} \textbf{123} (1996), no. 586, {\rm viii}+134 pp.
	
		
	\bibitem[Fe61]{Fe61}
	\textsc{J.M.G.~Fell},
	The structure of algebras of operator fields. 
	\textit{Acta Mathematica} \textbf{106} (1961), no.~3, 233-280.
	
 
	
	\bibitem[GT79]{GT79}
	\textsc{R.M.~Gillette, D.C.~Taylor}, 
	The minimal dense two-sided ideal of a $C^*$-algebra with continuous trace. 
	\textit{Math. Scand.} \textbf{44} (1979), no. 1, 201--206.
	

	
	\bibitem[Gl62]{Gl62}
	\textsc{J.~Glimm},  
	Families of induced representations. 
	\textit{Pacific J. Math.} \textbf{12}(1962), 885--911.
	
	\bibitem[GK22]{GK22}
	\textsc{M.~Goffeng, A.~Kuzmin}, Index theory of hypoelliptic operators on Carnot manifolds. 
	\textit{Preprint} arXiv:2203.04717.
	
	
	\bibitem[GS99]{GS99}
	\textsc{V.Y.~Golodets, S.D.~Sinel'shchikov}, 
	On the conjugacy and isomorphism problems for stabilizers of Lie group actions. 
	\textit{Ergodic Theory Dynam. Systems} \textbf{19} (1999), no. 2, 391--411.
	
	\bibitem[Gr77]{Gr77} 
	\textsc{Ph.~Green}, 
	$C^*$-algebras of transformation groups with smooth orbit space.
	\textit{Pacific J. Math.} \textbf{72} (1977), no.~1, 71--97. 
	
	\bibitem[Grl92]{Grl92}
	\textsc{G.~Gr\'elaud}, 
	On representations of simply connected nilpotent and solvable Lie groups. 
	Cours de DEA, 1992. 
	
	
		\bibitem[Gro73]{Gro73}
	\textsc{A.~Grothendieck}, 
	\textit{Topological vector spaces}. 
	Notes on Mathematics and its Applications. 
	Gordon and Breach Science Publishers, New York-London-Paris, 1973.
	
	\bibitem[dlH08]{dlH08}
	\textsc{P.~de la Harpe}, 
	Spaces of closed subgroups of locally compact groups. 
	\textit{Preprint} arXiv:0807.2030 [math.GR]. 
	
	\bibitem[HN12]{HN12}
	\textsc{J.~Hilgert, K.-H.~Neeb}, 
	\textit{Structure and geometry of Lie groups}. 
	Springer Monographs in Mathematics. Springer, New York, 2012.
	
		
	\bibitem[Pea75]{Pea75}
	\textsc{A.R.~Pears}, 
	\textit{Dimension theory of general spaces}. 
	Cambridge University Press, Cambridge, England-New York-Melbourne, 1975.
	
	\bibitem[P79]{P79}
	\textsc{G.K.~Pedersen}, 
	\textit{$C^{\ast} $-algebras and their automorphism groups}. 
	London Mathematical Society Monographs \textbf{14}. Academic Press, Inc., 
	London-New York, 1979.
	
	\bibitem[Pe85]{Pe85}
	\textsc{N.V.~Pedersen}, 
	Composition series of $C^*(G)$ and $C^\infty_c(G)$, where $G$ is a solvable Lie group.
	\textit{Invent. Math.} \textbf{82} (1985), no. 2, 191--206.
	
	\bibitem[Pu68]{Pu68}
	\textsc{L.~Puk\'anszky}, 
	On the unitary representations of exponential groups. 
	\textit{J. Functional Analysis} \textbf{2} (1968), 73--113. 
	
	
	\bibitem[RaRo98]{RaRo98}
	\textsc{I.~Raeburn, J.~Rosenberg}, 
	Crossed products of continuous-trace $C^*$-algebras by smooth actions. 
	\textit{Trans. Amer. Math. Soc.} \textbf{305} (1988), no. 1, 1--45.
	
	\bibitem[RaWi98]{RW98}
	\textsc{I.~Raeburn, D.P.~Williams}, 
	\textit{Morita equivalence and continuous-trace $C^*$-algebras}. 
	Mathematical Surveys and Monographs, 60. 
	American Mathematical Society, Providence, RI, 1998.
	
	\bibitem[Wi07]{Wi07}
	\textsc{D.P.~Williams}, 
	\textit{Crossed products of $C^*$-algebras}. 
	Mathematical Surveys and Monographs, 134. 
	American Mathematical Society, Providence, RI, 2007. 
	
\end{thebibliography}
\end{document}